\documentclass{article}
\usepackage{amssymb}
\usepackage{amsthm}
\usepackage[margin=1.25in]{geometry}
\usepackage{color}
\usepackage{caption}
\usepackage{subcaption}

\usepackage{mathtools}
\numberwithin{equation}{section}
\mathtoolsset{showonlyrefs=true}

\usepackage{amsfonts}
\usepackage{graphicx}%
\providecommand{\U}[1]{\protect\rule{.1in}{.1in}}
\newtheorem{theorem}{Theorem}[section]

\newtheorem{lemma}[theorem]{Lemma}

\newtheorem{proposition}[theorem]{Proposition}

\newcommand{\e}{\varepsilon}
\newcommand{\dx}{\, dx}
\newcommand{\dt}{\, dt}
\newcommand{\ds}{\, ds}

\newcommand{\I}{\mathcal{I}_\Omega}
\newcommand{\V}{\mathfrak{v}}

\newcommand{\gto}{\xrightarrow{\Gamma}}
\newcommand{\Iloc}{\I^{E_0,\delta}}
\newcommand{\LL}{\mathcal{L}}
\newcommand{\R}{\mathbb{R}}
\newcommand{\HH}{\mathcal{H}}

\title{A Note Regarding Second-Order $\Gamma$-limits for the Cahn--Hilliard Functional}
\author{Giovanni Leoni\\
        Carnegie Mellon University \\
        Pittsburgh, PA, USA
        \and Ryan Murray\\
        Penn State University \\
        University Park, PA, USA}

\begin{document}

\maketitle

\begin{abstract}
This note completely resolves the asymptotic development of order $2$ by $\Gamma$-convergence of the mass-constrained Cahn--Hilliard functional, by showing that one of the critical assumptions of the authors' previous work \cite{LeoniMurray} is unnecessary.
\end{abstract}



\section{Introduction}

In the recent paper \cite{LeoniMurray} we have solved in most cases a long standing open problem, namely, the asymptotic development by $\Gamma$-convergence of order $2$ of the Modica--Mortola or Cahn--Hilliard functional  (see \cite{GurtinMatano,Modica87,Sternberg88})
\begin{equation} 
F_\e(u) := \int_\Omega W(u) + \e^2 |\nabla u|^2 \dx,\qquad u \in H^1(\Omega), \label{F0Definition}
\end{equation}
subject to the mass constraint
\begin{equation} \label{massConstraintEquation}
\int_\Omega u \dx = m. 
\end{equation}
Here $\Omega \subset \mathbb{R}^n$, $2\le n \leq 7,$ is an open, connected, bounded set with
\begin{equation}\label{domainAssumptions}
\mathcal{L}^n(\Omega) =1 \quad \text{ and } \quad \partial \Omega \text{ is of class }  C^{2,\alpha}, \quad \alpha \in (0,1],  
\end{equation}
and the double-well potential $W: \mathbb{R} \to [0,\infty)$ satisfies:
\begin{align}
\label{W_Smooth}&W \text{ is of class $C^2(\mathbb{R}\backslash \{a,b\})$ and has precisely two zeros at } a<b, \\
\label{WPrime_At_Wells}&\lim_{s \to a} \frac{W''(s)}{|s-a|^{q-1}} = \lim_{s \to b}\frac{W''(s)}{|s-b|^{q-1}} := \ell > 0, \quad q \in (0,1] ,\\
\label{W_Number_Zeros} &\text{ $W'$ has exactly 3 zeros  at $a<c<b$,} \quad W''(c)<0, \\
\label{WGurtin_Assumption}& \liminf_{|s| \to \infty} |W'(s)| > 0.
\end{align}
 We assume that the mass $m$ in \eqref{massConstraintEquation} satisfies
\begin{equation} \label{originalMassRange}
a<m<b.
\end{equation}

We recall that given a metric space $X$ and a family of functions $\mathcal{F}_\e : X \to \overline{ \mathbb{R}}, \hspace{1mm}\e > 0$, an \emph{asymptotic development} of \emph{order k}
\begin{equation}\label{asymptoticDevelopmentDefinition}
\mathcal{F}_\e = \mathcal{F}^{(0)} + \e\mathcal{F}^{(1)} + \dots + \e^k\mathcal{F}^{(k)} + o(\e^k)
\end{equation}
holds if there exist functions $\mathcal{F}^{(i)}:X\to\overline{\mathbb{R}}$, $i = 0,1,\dots,k$, such that the functions
\begin{equation} \label{higherOrderFunctionalDefinition}
\mathcal{F}^{(i)}_\e := \frac{\mathcal{F}^{(i-1)}_\e-\inf_X \mathcal{F}^{(i-1)}}{\e}
\end{equation}
are well-defined and
\begin{equation} \label{higherOrderGammaConvergenceDefinition}
\mathcal{F}^{(i)}_\e \xrightarrow{\Gamma} \mathcal{F}^{(i)},
\end{equation}
where $\mathcal{F}^{(0)}_\e := \mathcal{F}_\e$ and $\overline{\mathbb{R}}$ is the extended real line (see \cite{AnzellottiBaldo}). 

In our case $X:= L^1(\Omega)$ and we define
\begin{equation}\label{formalFunctionalDefinition}
\mathcal{F}_\e(u) := \begin{cases}
F_\e(u) &\text{ if } u \in H^1(\Omega) \text{ and } (\ref{massConstraintEquation})\text{  holds},\\
\infty &\text{ otherwise in } L^1(\Omega),
\end{cases}
\end{equation}
where $F_\e$ is the functional in \eqref{F0Definition}. It is well-known (see \cite{CarrGurtinSlemrod},  \cite{Modica87}, \cite{Sternberg88}) that, under appropriate assumptions on $\Omega$ and $W$, the $\Gamma$-limit $\mathcal{F}^{(1)}$ of order $1$ (see \eqref{higherOrderFunctionalDefinition} and \eqref{higherOrderGammaConvergenceDefinition}) of \eqref{formalFunctionalDefinition} is given by
\begin{equation} \label{firstOrderFormalDefinition}
\mathcal{F}^{(1)}(u) := \begin{cases}
2c_W\operatorname*{P}(\{u=a\};\Omega) &\text{ if } u \in BV(\Omega;\{a,b\}) \text{ and \eqref{massConstraintEquation} holds},\\
\infty &\text{ otherwise in } L^1(\Omega),
\end{cases}
\end{equation}
where $\operatorname*{P}(\cdot;\Omega)$ is the perimeter in $\Omega$ (see \cite{AmbrosioFuscoPallara,EvansGariepy,Ziemer}), $a,b$ are the wells of $W$, and 
\begin{equation} \label{c0Definition}
c_W := \int_a^b W^{1/2}(s)  \ds.
\end{equation}
Hence, $u \in BV(\Omega;\{a,b\})$ is a minimizer of the functional $\mathcal{F}^{(1)}$ in \eqref{firstOrderFormalDefinition} if and only if the set $\{u=a\}$ is a solution of the classical \emph{partition problem} 
\begin{equation}\label{eqn:iso-func}
 \I(\V):= \min \{P(E;\Omega) : E \subset \Omega, \mathcal{L}^n(E) = \V\}
\end{equation}
at the value $\V=\V_m$, where (see \eqref{domainAssumptions})
\begin{equation}\label{partitionProblemMassConstraint}
\V_m := \frac{b - m}{b-a} .
\end{equation}
When $\Omega$ is bounded and of class $C^2$,  minimizers $E$ of \eqref{eqn:iso-func} exist, have constant generalized mean curvature $\kappa_E$, intersect the boundary of $\Omega$ orthogonally, and their singular set is empty if $n\leq 7$, and has dimension at most $n-8$ if $n \geq 8$ (see \cite{GMT83,GruterBoundaryRegularity, MaggiBook,SternbergZumbrunIsoPer}). Here and in what follows we use the convention that $\kappa_E$ is the average of the principal curvatures taken with respect to the outward unit normal to $\partial E$.

Under the hypothesis that the \emph{isoperimetric function} $\V\mapsto \I(\V)$ satisfies the Taylor expansion
\begin{equation} \label{isoFunctionSmooth}
\I(\V) = \I(\V_m) + \I'(\V_m)(\V-\V_m) + O(|\V-\V_m|^{1+\beta})
\end{equation}
for all $\V$ close to $\V_m$ and for some $\beta \in (0,1]$, in \cite{LeoniMurray} we proved the following theorems (see also \cite{DalMasoFonsecaLeoni}).
\begin{theorem}[\cite{LeoniMurray}]\label{mainThm1}
        Assume that $\Omega,m,W$ satisfy hypotheses \eqref{domainAssumptions}-\eqref{originalMassRange} with $q =1$ and that \eqref{isoFunctionSmooth} holds. Then
        \begin{equation}\label{mainThmEqn1}
        \mathcal{F}^{(2)}(u) = \frac{2c_W^2(n-1)^2}{W''(a)(b-a)^2}\kappa_u^2 + 2(c_{\operatorname*{sym}}+c_W\tau_u)(n-1)\kappa_u\operatorname*{P}(\{u=a\};\Omega) 
        \end{equation}
        if $u \in BV(\Omega;\{a,b\})$ is a minimizer of the functional $\mathcal{F}^{(1)}$ in \eqref{formalFunctionalDefinition} and $\mathcal{F}^{(2)}(u) = \infty$ otherwise in $L^1(\Omega)$.
\end{theorem}

\begin{theorem}[\cite{LeoniMurray}]\label{mainThm2}
        Assume that $\Omega,m,W$ satisfy hypotheses \eqref{domainAssumptions}-\eqref{originalMassRange} with $q \in (0,1)$  and that \eqref{isoFunctionSmooth} holds. Then
        \begin{equation}\label{mainThmEqn2}
        \mathcal{F}^{(2)}(u) = 
        2(c_{\operatorname*{sym}}+c_W\tau_u)(n-1)\kappa_u\operatorname*{P}(\{u=a\};\Omega) 
        \end{equation}
                if $u \in BV(\Omega;\{a,b\})$ is a minimizer of the functional $\mathcal{F}^{(1)}$  and $\mathcal{F}^{(2)}(u) = \infty$ otherwise in $L^1(\Omega)$.
\end{theorem}

Here $\kappa_u$ is the constant mean curvature of the set $\{u=a\}$, 
\begin{equation}\label{c1Definition}
c_{\operatorname*{sym}} := \int_\mathbb{R} W(z(t))t  \dt,
\end{equation}
where $z$ is the solution to the Cauchy problem
\begin{equation} \label{profileCauchyProblem}
\begin{cases} z'(t) = \sqrt{W(z(t))} \quad &\text{ for } t \in \mathbb{R}, \\ z(0) = c,\quad &z(t) \in [a,b],\end{cases}
\end{equation}
with $c$ being the central zero of $W'$ (see \eqref{W_Number_Zeros}), and $\tau_u \in \mathbb{R}$ is a constant such that
\begin{equation}\label{deltaUDefinition}
\operatorname*{P}(\{u=a\};\Omega)\int_\mathbb{R} z(t - \tau_u) - \operatorname*{sgn}\nolimits_{a,b}(t)  \dt = \frac{2c_W(n-1)}{W''(a)(b-a)}\kappa_u
\end{equation}
if $q=1$ in \eqref{WPrime_At_Wells} and
\begin{equation}\label{deltaUDefinitionCase2}
\int_\mathbb{R} z(t - \tau_u) - \operatorname*{sgn}\nolimits_{a,b}(t)  \dt = 0
\end{equation}
if $q\in(0,1)$ in \eqref{WPrime_At_Wells}, where
\begin{equation} \label{sgnabDefinition}
\operatorname*{sgn}\nolimits_{a,b}(t) := \begin{cases}
a &\text{ if } t \leq 0 ,\\
b &\text{ if } t > 0.
\end{cases}
\end{equation}

The assumption \eqref{isoFunctionSmooth} is known to hold at a.e. $\V_m$, or, equivalently, for a.e. mass $m\in (a,b)$, since $\I$ is semi-concave \cite{MurrayRinaldi}. However, in the case that the isoperimetric function is differentiable at $\V_m$ the mean curvature of the interface of minimizers is completely determined since (see Chapter 17 in\cite{MaggiBook}) 
 \begin{equation} \label{isoPerDerivative}
 \I'(\V_m) = (n-1)\kappa_E
 \end{equation}
 for every minimizer $E$ of \eqref{eqn:iso-func} with $\V = \V_m$. 
Hence Theorems \eqref{mainThm1} and \eqref{mainThm2} \emph{do not provide a selection criteria for minimizers.}
Indeed, the case of two global minimizers of the partition problem \eqref{eqn:iso-func} with different mean curvatures is excluded  by \eqref{isoFunctionSmooth}.

The purpose of this note is to \emph{remove the assumption that $\I$ is regular at $\V_m$}. Specifically, the theorem that we prove is the following:

\begin{theorem}\label{theorem main}
        Theorems \eqref{mainThm1} and \eqref{mainThm2} continue to hold without assuming \eqref{isoFunctionSmooth}. 
\end{theorem}

The $\limsup$ portion of this theorem was already established in \cite{LeoniMurray}, see Remark 5.5 in the same. Thus this work focuses on proving the $\Gamma$-$\liminf$ inequality.

Besides its intrinsic interest, Theorem \ref{theorem main}  has important applications in the study of the speed of motion of the mass-preserving Allen--Cahn equation 
\begin{equation} \label{NLAC}
\begin{cases} \partial_t u_\e = \e^2 \Delta u_\e - W'(u_\e) + \e \lambda_\e &\text{ in } \Omega \times [0,\infty), \\
\  \frac{\partial u_\e}{\partial \nu} = 0\qquad &\text{ on } \partial \Omega \times [0,\infty),\\
\ \ \, u_\e = u_{0,\e} \qquad &\text{ on } \Omega \times \{0\}
\end{cases}
\end{equation}
and of the Cahn--Hilliard equation
\begin{equation} \label{CHivp}
\left \{
\begin{aligned}
\partial_t u_{\e} &= - \Delta (\e^2 \Delta u_{\e} - W'(u_{\e}))  &\mbox{in } \Omega \times (0,\infty), \\
\frac{\partial u_{\e}}{\partial \nu} &= \frac{\partial }{\partial \nu}(\e^2 \Delta u_{\e} - W'(u_{\e})) = 0 &\mbox{ on } \partial \Omega \times [0,\infty), \\
u_{\e} &= u_{0,\e} &\mbox{on } \Omega \times \{0 \}
\end{aligned}
\right.
\end{equation}
in dimension $n\ge2$.

In what follows we say that a measurable set $E_0 \subset \Omega$ is a \emph{volume--constrained local perimeter minimizer} of $P(\cdot, \Omega)$ if there exists $\rho > 0$ such that
\[
P(E_0;\Omega) = \inf\left\{ P(E;\Omega): E \subset \Omega \text{ Borel, } \LL^n(E_0) = \LL^n(E), \, \LL^n(E_0 \ominus E)  < \rho\right\},
\]
where $\ominus$ denotes the symmetric difference of sets. 
We define
\begin{equation}
\label{u E0}
u_{E_0}:=a\chi_{E_0}+b\chi_{\Omega\setminus E_0}.
\end{equation}
The following theorem significantly improves Theorems 1.2,   1.4, 1.7, and 1.9  in \cite{MurrayRinaldi}. In particular, the assumption that the local minimizer $E_0$ has positive second variation (see Theorem 1.9 in \cite{MurrayRinaldi}) is no longer needed.

\begin{theorem} \label{theorem ac}
        Assume that $\Omega,m,W$ satisfy hypotheses \eqref{domainAssumptions}-\eqref{originalMassRange} with $q =1$, and let $E_0$ be a volume--constrained local perimeter minimizer with $\mathcal{L}^n(E_0) = \V_m$. Assume that $u_{0,\e} \in L^\infty(\Omega)$ satisfy       
        \begin{equation}
        \int_\Omega u_{0,\e} \dx = m,\qquad     u_{0,\e} \to u_{E_0} \text{ in } L^2(\Omega) \text{ as } \e \to 0^+
        \end{equation}
        and
        \begin{equation}
        \mathcal{F}_\e(u_{0,\e}) \leq \e        \mathcal{F}^{(1)}(u_{E_0}) + C\e^2
        \end{equation}
        for some $C > 0$. Let $u_\e$ be a solution to \eqref{NLAC}. 
        Then, for any $M > 0$,
        \[
        \sup_{0 \leq t \leq M\e^{-1}} || u_{\e}(t) - u_{E_0}||_{L^1(\Omega)} \to 0 \ \text{as} \ \e \to 0^+.
        \]
\end{theorem}
The following theorem improves upon Theorem 1.4 in \cite{MurrayRinaldi}.

 \begin{theorem} \label{globalMotion}
        Assume that $\Omega,m,W$ satisfy hypotheses \eqref{domainAssumptions}-\eqref{originalMassRange} with $q =1$ and that there exists a constant $C_1 > 0$ so that
        \begin{equation} \label{w9}
        |W'(s)| \leq C_1 |s|^p + C_1,
        \end{equation}
	where $p = \frac{n}{n-2}$ for $n \geq 3$, and $p > 0$ for $n = 1,2$. Let $E_0$ be a volume--constrained global perimeter minimizer with $\LL^n(E_0) = \V_m$. Assume that $u_{0,\e} \in L^2(\Omega)$ satisfy
        \begin{equation} \label{newL1conv2}
                \int_\Omega u_{0,\e} \dx = m,\qquad     u_{0,\e} \to u_{E_0} \text{ in } (H^1(\Omega))' \text{ as } \e \to 0^+
        \end{equation}
        and
        \begin{equation} \label{energyUpperCH}
        \mathcal{F}_{\e}(u_{0,\e}) \leq \mathcal{F}_0(u_{E_0})\e  + C\e^2
        \end{equation}
        for some $C > 0$. Let $u_\e$ be a solution to \eqref{CHivp}. Then, for any $M > 0$,
        \begin{equation} \label{slowCH}
        \sup_{0 \leq t \leq M\e^{-1}} ||u_\e - u_{E_0}||_{(H^1(\Omega))'} \to 0 \ \text{as} \ \e \to 0^+.
        \end{equation}
 \end{theorem}

\section{Localized Isoperimetric Function}

One of the central ideas in \cite{LeoniMurray} was the development and use of a generalized P\'olya--Szeg\H o inequality to reduce the second-order $\Gamma$-$\liminf$ of \eqref{F0Definition} to a one-dimensional problem. This generalized P\'olya--Szeg\H o inequality relied on comparing the perimeter of the level sets of arbitrary functions with values of $\I$. On the one hand, this approach is simple and very general. On the other hand, it is clearly not sharp in our setting because the minimizers of \eqref{eqn:iso-func} may be widely separated in $L^1$, while the transition layers we are considering are known to converge in $L^1$. Hence, the isoperimetric function may be too pessimistic in estimating the perimeter of the level sets of transition layers.

%

In light of this, following \cite{MurrayRinaldi}, we use instead a localized version of the isoperimetric function. Specifically, given a set $E_0$, and some $\delta > 0$, we  define the \emph{local isoperimetric function} of parameter $\delta$ about the set $E_0$ to be
\begin{equation} \label{labelstar3}
\Iloc(\V) := \inf \{ P(E, \Omega):\, E \subset \Omega \text{ Borel, } \LL^n(E) =\V,\, \alpha(E_0,E) \leq \delta    \},
\end{equation}
where
\begin{equation} \label{labelsharp3}
\alpha(E_1,E_2) := \min\{ \LL^n(E_1 \setminus E_2), \LL^n(E_2 \setminus E_1)   \}
\end{equation}
for any Borel sets $E_1,E_2 \subset \Omega$.

The following proposition, which connects the definition of $\Iloc$ with $L^1$ convergence, can be found in \cite{MurrayRinaldi}. We present its proof for the convenience of the reader.

\begin{proposition} \label{levelSetsProposition} Let $\Omega \subset \R^n$ be an open set, $E_0 \subset \Omega$ be a Borel set and let $u_{E_0}$ be as in \eqref{u E0}. Then
\begin{equation} \label{3star}
\alpha(E_0,\{u \leq s\}) \leq \delta
\end{equation}
for all $u \in L^1(\Omega)$ such that
\begin{equation} \label{iso1}
\|u-u_{E_0}\|_{L^1} \leq (b-a)\delta,
\end{equation}
and for every $s \in \R$, where the function $\alpha$ is given in \eqref{labelsharp3}.
\end{proposition}

\begin{proof}
Fix $\delta > 0$, and for $s \in \R$ define $F_s := \{ x \in \Omega: u(x) \leq s   \}$.
If $a < s < b$, then by \eqref{iso1},
\[
\begin{aligned}
(b-a)\delta &\geq \int_{ F_s \setminus E_0 }  |u-u_{E_0}|\dx  + \int_{ E_0 \setminus F_s}  |u-u_{E_0}| dx  \\
&\geq (b-s) \LL^n(F_s \setminus E_0) + (s-a) \LL^n(E_0 \setminus F_s) \geq (b-a)\alpha(E_0,F_s),
\end{aligned}
\]
so that \eqref{3star} is proved in this case. If $s \geq b$, again by \eqref{iso1},
\[
(b-a)\delta \geq \int_{E_0 \setminus F_s} |u - u_{E_0}| dx \geq (s-a) \LL^n(E_0 \setminus F_s) \geq (b-a)\alpha(E_0,F_s).
\] 
The case $s \leq a$ is analogous.
\end{proof}

By construction, we know that $\Iloc \geq \I$. Furthermore, by BV compactness and lower-semicontinuity, and the fact that $\alpha$ is continuous in $L^1$, we have that $\Iloc$ is lower semi-continuous. The next proposition establishes a stronger type of regularity for $\Iloc$.

\begin{proposition} \label{Prop:Semi-Concave} Assume that $\Omega$ satisfies \eqref{domainAssumptions} and let $E_0\subset\Omega$ be a local volume-constrained perimeter minimizer, with $  \LL^n(E_0)=\V_m$. Then for $\delta$ small enough there exists a neighborhood $J_\delta$ of $\V_m$ so that $\Iloc$ is semi-concave on $J_\delta$.
\end{proposition}

Before proving this proposition, we state and prove a technical lemma. 
In what follows we say that an open set $U\subset\mathbb{R}^n$ has \emph{piecewise $C^2$ boundary} if $\partial U$ can be written as the union of finitely many connected $(n-1)$-dimensional manifolds  with boundary of class $C^2$ up to the boundary, with pairwise disjoint relatively interiors. 

\begin{lemma} \label{Lem:Isomorphism}
   Let $U = \Omega \backslash \overline{E_0}$ for some volume constrained perimeter minimizer $E_0$. Given $\tau>0$, let 
\begin{equation}\label{Ut}
U_\tau:=\{x\in U : d(x,\R^n \backslash U) > \tau\}.
\end{equation}
Then there exist a $A>0$ and $C_1,C_2 >0$ so that for all $\tau$ sufficiently small and all $\V \in (C_1 \tau ,A)$,
\[
\mathcal{I}_{U_\tau}(\V) \geq C_2 \V^{(n-1)/n}.
\]
\end{lemma}

\begin{proof}
   We remark that the the boundary of $U$ will have piecewise $C^2$ with components that meet transversally. Furthermore the components of the boundary of $U$ can be locally extended without intersecting $U$. 
  
{\bf Step 1:} We begin by constructing a $C^1$ vector field $T$ which points into the domain $U$. 

Let $M_i$, $i=1,\ldots,m$,  be the finitely many connected $(n-1)$-dimensional manifolds of class $C^2$ with boundary whose union gives $\partial U$. Extend each $M_i$ in such a way that $M_i$ is a subset of  the boundary of an open set $V_i$ of class $C^2$ with  $V_i \cap U=\emptyset$. Set $F_i=\partial V_i$. 
%
Next we extend the normal vector field $\nu_{F_i}$ to a vector field $T_i\in  C^{1}(\mathbb{R}^n;\mathbb{R}^n)$.
If $M_i$ and $M_j$ intersect transversally, then for $x\in \partial M_i\cap\partial M_j$ we have $T_i(x)\cdot T_j(x)=\nu_{F_i}(x)\cdot \nu_{F_j}(x)=0$ and thus by continuity we can find $\tilde \rho>0$ such that $| T_i(x)\cdot T_j(x)|\le \frac1{2m}$ for all $x$ in a $\tilde \rho$-neighborhood (denoted by $U_{i,j}$) of $\partial M_i\cap\partial M_j$. By taking $\tilde \rho$ even smaller, if necessary, we can assume that the same $\tilde \rho$ works for all $i$ and $j$ such that $M_i$ and $M_j$ intersect transversally. Next, set
\[
d_0: = \min_{i \neq j} d(M_i \backslash U_{i,j}, M_j \backslash U_{i,j})>0
\]
and let $\rho := 1/2 \min(\tilde \rho, d_0)>0$.

We then choose smooth cutoff functions $\varphi_i$ which are valued $1$ on $M_i$ and $0$ at distance $\rho/2$ from the same sets and consider the vector field $T: =\sum_{i=1}^{m} \varphi_i T_i$. Note that $T\in  C^{1}(\mathbb{R}^n;\mathbb{R}^n)$, with $\Vert T\Vert_\infty \leq C$ and $\Vert \nabla T\Vert_\infty\le C$ for some constant $C>0$.

We claim that 
\begin{equation}\label{normal}
  T(x) \cdot \nabla  d_{V_i}(x) \ge 1/4
\end{equation}
for all points $x\in U$ in a $\rho_0$-neighborhood of $F_i$, where $d_{V_i}$ is  the signed distance to the set $V_i$ enclosed by $F_i$. 
By Theorem 3 in \cite{Krantz-Parks} we have that $d_{V_i}$ is a $C^2$ function in a neighborhood of $F_i$.

Suppose $x\in M_i$. Then $T_i(x)=\nu_{F_i}(x)=\nabla  d_{V_i}(x)$, and so 
\begin{equation}\label{normal2}
  T(x) \cdot \nabla  d_{V_i}(x) =1+ \sum_{j\ne i}^{m} \varphi_j(x) T_j(x)\cdot T_i(x).
\end{equation}
If $x$ is in $\rho$-neighborhood of $\partial M_i\cap\partial M_j$, then $T_j (x)\cdot T_i(x)\ge -\frac{1}{2m}$ otherwise $\varphi_j(x)=0$. Thus, in both cases $ T(x) \cdot \nabla d_{V_i}(x) \ge\frac{1}{2}$.
By continuity of $T$ and $\nabla  d_{V_i} $, the inequality \eqref{normal2} implies that \eqref{normal} holds in a neighborhood of $M_i$.

{\bf Step 2:} We consider the flow along $T$, meaning that for $x\in\mathbb{R}^n$ we take the initial value problem
\[
\left\{
\begin{array}
[c]{l}%
\frac{d\Psi}{dt}(t)=T(\Psi(t)),\\
\Psi(0)=x.
\end{array}
\right.
\]
Since $T$ is Lipschitz continuous, there exists a  unique global solution $\Psi$
defined for all $t\in\mathbb{R}$. To highlight the dependence on $x$ we write
$\Psi(\cdot,x)$ and we define $\Psi_{t}(x):=\Psi(t,x)$. Let $U^t: = \Psi_t(U)$. By construction $\Psi_t$ satisfies
\[
(1-Ct) |x-y| \leq |\Psi_t(x) - \Psi_t(y)| \leq (1+Ct) |x-y|.
\]
This implies that for any set $E \subset U^t$ of finite perimeter, 
\begin{align}\label{bi-lipschitz}
(1-Ct)^n \LL^n(E) &\leq \LL^n(\Psi_t^{-1}(E)) \leq (1+Ct)^n \LL^n(E),\\(1-Ct)^{n-1} P(E;U^t) &\leq P(\Psi_t^{-1}(E);U) \leq (1+Ct)^{n-1} P(E;U^t).
\end{align}

We claim that $U_{\tau}\subset U^{c_{3}\tau}$, where $c_{3}:=1/\Vert T\Vert_{\infty}$. To see this, let $y\in U_{\tau}$. For every
$t\in\mathbb{R}$ we have $|\Psi_{t}(y)-y|\leq|t|\Vert T\Vert_{\infty}$, and
so
\[
d(\Psi_{t}(y),\mathbb{R}^{n}\setminus U)\geq d(y,\mathbb{R}^{n}\setminus
U)-|\Psi_{t}(y)-y|>\tau-|t|\Vert T\Vert_{\infty}\geq0
\]
provided $|t|\leq\tau/\Vert T\Vert_{\infty}$. In turn, $\Psi_{t}(y)\in U$ for
$|t|\leq\tau/\Vert T\Vert_{\infty}$. Define  $x_{\tau}:=\Psi_{-c_{3}\tau}(y)$. and consider the function $\Psi
(\cdot,x_{\tau})$. Since the system of differential equations is autonomous
and solutions are unique, we have that $\Psi_{c_{3}\tau}(x_{\tau})=\Psi
_{c_{3}\tau}(\Psi_{-c_{3}\tau}(y))=y$, which shows that $y\in U^{c_{3}\tau
}=\Psi_{c_{3}\tau}(U)$.


Next, we claim that $U^{c_3\tau} \subset U_{c_4 \tau}$ for all $\tau$ sufficiently small, and for some constant $c_4$. Let $x \in U$ be in a $\rho_0/2$ neighborhood of $M_i$, where $\rho_0$ was given in  Step 1. Since $d_{V_i}$ is $C^2$, by the chain rule we may write
\begin{align*}
  d_{V_i}(\Psi_t(x)) &= d_{V_i}(x) + \int_0^t \nabla d_{V_i}(\Psi_s(x))\cdot T(\Psi_s(x)) \,ds\\
  &\geq d_{V_i}(x) + \frac{t}{4} \geq d(x,\R^n \backslash U) + \frac{t}{4},
\end{align*}
where we have used \eqref{normal}, and where we have assumed that $t < \frac{\rho_0}{2\|T\|_\infty}$. As this is true for all $i$, and as $d(x,\R^n\backslash U) = \min_i d_{V_i}(x)$ for $x\in U$, we find that
\[
d(\Psi_t(x),\R^n \backslash U) \geq d(x, \R^n \backslash U) + t/4
\]
for $x$ near $\partial U$ and for $t$ sufficiently small. This proves the claim for $x$ close to the boundary, and for $x$ far away from the boundary there is nothing to prove.

In summary, we know that $U_\tau \subset U^{c_3\tau} \subset U_{c_4 \tau}$, as along as $\tau$ is sufficiently small, for $c_3,c_4$ independent of $\tau$.  These two inclusions,  along with \eqref{bi-lipschitz}, imply that for any set $E$ of finite perimeter we have that
\[
  P(E;U_\tau) \geq P(E;U^{c_3/c_4\tau})
\]
and that $\LL^n(U_\tau \backslash U^{c_3/c_4 \tau})\le  c_5 \tau$, with $c_5>0$ independent of $\tau$. 

Finally, let $E\subset U_\tau$  be a set of finite perimeter satisfying $\LL^n(E) > 2 c_5 \tau$. By \eqref{bi-lipschitz}, the previous inequalities, and the isoperimetric inequality  (which applies as $U$ must be Lipschitz) we have that
\begin{align*}
  P(E;U_\tau) &\geq P(E;U^{c_3/c_4\tau}) \\
&\geq CP(\Psi_{c_3 /c_4\tau}^{-1}(E \cap U^{c_3/c_4\tau});U) \\
&\geq C\left( \LL^n(\Psi_{c_3 /c_4\tau}^{-1}(E\cap U^{c_3/c_4\tau}))\right)^{(n-1)/n} \\
&\geq C\left( \LL^n(E\cap U^{c_3/c_4\tau})\right)^{(n-1)/n} \\
&\geq C \left( \LL^n(E) - c_5 \tau\right)^{(n-1)/n} \geq C \LL^n(E)^{(n-1)/n}.
\end{align*}
This completes the proof.

\end{proof}

Now we prove Proposition \ref{Prop:Semi-Concave}.

\begin{proof}[Proof of Proposition \ref{Prop:Semi-Concave}]
Let $E_{\hat \V}$ be a minimizer of
\begin{equation}\label{def:iloc-min-prob}
\min \{ P(E;\Omega)  :\, E \subset \Omega \text{ Borel, } \LL^n(E) = \hat \V,\, \alpha(E_0,E) \leq \delta  \},
\end{equation}
 with $\hat \V \in J_\delta = (\V_m - \delta/2, \V_m + \delta/2)$. Since $\partial E_0$ is regular and intersects the boundary of $\Omega$ orthogonally, we know that 
\begin{equation} \label{eqn:iso-lower-bound}
\mathcal{I}_{\Omega \cap E_0}(\V) \geq C \V^{\frac{n-1}{n}}, \qquad \mathcal{I}_{\Omega \backslash E_0}(\V) \geq C \V^{\frac{n-1}{n}}
\end{equation}
for all $\V$ sufficiently small (see, e.g., Corollary 3 in Section 5.2.1 of \cite{MazyaBook}). We pick $\delta$ small enough that \eqref{eqn:iso-lower-bound} holds for $\V \in (0, 2 \delta)$.

Next, we claim that we can construct a smooth function $\phi_{\hat \V}$ defined on a neighborhood of $\hat \V$ so that
\begin{equation} \label{eqn:touching-function}
\phi_{\hat \V}(\hat \V) = \Iloc(\hat \V),\quad \phi_{\hat \V}(\V) \geq \Iloc(\V), \quad \phi_{\hat \V}'' \leq C,
\end{equation}
where $C$ does not depend on $\hat \V$, but may depend on $\delta$.

To prove this claim, we consider two different cases. First, suppose that $\alpha(E_{\hat \V}, E_0) < \delta$. Then by \eqref{labelstar3}, $E_{\hat \V}$ is actually a volume-constrained local  perimeter minimizer, and hence we can prove \eqref{eqn:touching-function} by using a normal perturbation and the fact that $\partial\Omega$ is smooth, see Lemma 4.3 in \cite{MurrayRinaldi} for details.

Now suppose that $\alpha(E_{\hat \V}, E_0) = \delta$. In view of \eqref{labelsharp3} we may assume, without loss of generality, that $\alpha(E_{\hat \V}, E_0) = \LL^n(E_0 \backslash E_{\hat \V})$ (the opposite case is analogous). Hence, we may locally perturb $E_{\hat \V}$ inside the set $\Omega \backslash E_0$ without increasing the value of $\alpha(E_{\hat \V}, E_0)$. In particular, by \eqref{labelstar3},  $E_{\hat \V}\backslash E_0$ is a local minimizer of the problem
\[
\min \{P(E;\Omega \backslash E_0) : E \subset \Omega\backslash E_0 \text{ Borel, } \LL^n(E)= \LL^n(E_{\hat \V} \backslash E_0)\}.
\]
Hence, by \cite{GMT83}, $\partial E_{\hat \V} \cap (\Omega \backslash E_0)$ is analytic.

We note that
\begin{align*}
 \delta - \LL^n(E_{\hat \V} \backslash E_0) &= \LL^n(E_0 \backslash E_{\hat \V}) - \LL^n(E_{\hat \V} \backslash E_0) \\
 &=  \LL^n(E_0)  - \LL^n(E_{\hat \V})=\V_m-\hat \V \in (-\delta/2, \delta/2).
\end{align*}
Hence, we know that $\LL^n(E_{\hat \V} \backslash E_0) \in [\delta, \frac{3 \delta}{2}]$. Since $E_0$ is a local volume constrained perimeter minimizer by \cite{GMT83} and \cite{GruterBoundaryRegularity}, its boundary is smooth inside $\Omega$ and intersects $\partial \Omega$ transversally. In particular, it may only have finitely many connected components, and hence by selecting $\delta$ sufficiently small we may assume that $\partial E_{\hat \V} \cap (\Omega \backslash E_0)$ is non-empty.

Next, let $U: = \Omega \backslash \overline{E_0}$. Let $\tilde d \in C^\infty(\R^n \backslash \partial U)$ be a regularized distance function from $\R^n \backslash U$, satisfying the properties
\begin{equation}
C_1 \leq \frac{\tilde d(x)}{d(x,\R^n \backslash U)} \leq C_2 \quad \text{ for $x \in U$}, \quad\|\nabla \tilde d\|_\infty \leq C, \label{Eqn:Signed-Distance-Properties}
\end{equation}
where $d(x,\R^n \backslash U)$ is the signed distance function.
Such a regularized distance function, as well as the aforementioned properties, is constructed in \cite{SteinBookSingularIntegral}.

Let $\phi_{\tau}: \R \to \R^+$ be a smooth function satisfying $\phi_{\tau}(s) = 0$ for all $s <  \tau/2 $, $\phi_{\tau}(s) = 1$ for all $s > \tau $, with $\phi_{\tau}$ strictly increasing for $\tau/2 < s < \tau$, and $\|\phi_{\tau} '\|_\infty \leq \frac{C}{\tau}$ with  $\tau$ to be chosen. We define $\Phi_{\tau}(x) := \phi_{\tau}(\tilde d(x))$.

Let $T \in C_c^\infty(U;\R^n)$ be an extension of the vector field $\Phi_{\tau} \nu_{\partial E_{\hat \V}}$. Define a one parameter family of diffeomorphisms given by $f_t(x) = x + tT(x)$, where $t$ is sufficiently small. Note that $f_t(x) = x$ for all $x\in \overline{E_0}$ and all $t$ sufficiently small. Hence by \eqref{labelstar3} the sets $f_t(E_{\hat \V})$ satisfy $P(f_{t}(E_{\hat \V});\Omega)\ge  \Iloc(\LL^n(f_t(E_{\hat \V})))$ since $\alpha(E_{\hat \V}, E_0) = \LL^n(E_0 \backslash E_{\hat \V})$. Using the formulas in Chapter 17 of \cite{MaggiBook}, there exists a function $\phi_{\hat \V} = P(f_{t(\V)}(E_{\hat \V});\Omega)$ such that for all $\V$ in a neighborhood of $\hat \V$: 
\begin{align}
\phi_{\hat \V}(\hat \V) &= \Iloc(\hat \V),\qquad \phi_{\hat \V}(\V) \geq \Iloc(\V), \\
\phi_{\hat \V}''( \hat \V) &= \frac{\int_{\partial E_{\hat \V}} |\nabla_{\partial E_{\hat \V}} \Phi_{\tau}|^2 - \Phi_{\tau}^2 \|A_{E_{\hat \V}}\|^2 \,d \HH^{n-1}}{\left(\int_{\partial E_{\hat \V}} \Phi_{\tau} \,d \HH^{n-1}\right)^2},
\end{align}
 where $\|A_{E_{\hat \V}}\|$ is the Frobenius norm of the second fundamental form of the boundary of $E_{\hat \V}$, and where the mapping $t(\V) \to \V$ is a smooth, increasing map with $t(0) = 0$. The second derivative formula can be proved as in \cite{MaggiBook}, \cite{SternbergZumbrunIsoPer}. 

In order to prove \eqref{eqn:touching-function} we thus only need to prove that
\begin{equation}\label{Eqn:Second-Derivative-Bound}
\frac{\int_{\partial E_{\hat \V}} |\nabla_{\partial E_{\hat \V}} \Phi_{\tau}|^2 \,d \HH^{n-1}}{\left(\int_{\partial E_{\hat \V}} \Phi_{\tau} \,d \HH^{n-1}\right)^2}\le  C.
\end{equation}
To this end, using \eqref{Eqn:Signed-Distance-Properties} and the fact that $\|\phi_{\tau} '\|_\infty \leq \frac{C}{\tau}$ we have that
\begin{equation}\label{Eqn:numerator-bound}
  \int_{\partial E_{\hat \V}} |\nabla_{\partial E_{\hat \V}} \Phi_{\tau}|^2 \,d \HH^{n-1} \leq \frac{C}{\tau^2} P(E_{\hat \V};\Omega) \le  \frac{C}{\tau^2}.
\end{equation}
On the other hand, denoting the set $\tilde U := \{\Phi_{\tau} \geq 1\} = \{\tilde d \geq \tau\}$, we have that
\begin{equation}\label{Eqn:denominator-bound}
\int_{\partial E_{\hat \V}} \Phi_{\tau} \,d \HH^{n-1} \geq \int_{\partial E_{\hat \V} \cap \tilde U}  \,d \HH^{n-1} = P(\partial E_{\hat \V} ; \tilde U).
\end{equation}
By \eqref{Eqn:Signed-Distance-Properties} and the fact that $U$ has Lipschitz boundary, we have that
\[
\LL^n(U \backslash \tilde U) \leq \LL^n(\{x : 0 \leq d(x,\R^n \backslash U) \leq C_2 \tau\}) \leq C \tau.
\]
Using the notation \eqref{Ut} we also have, by \eqref{Eqn:Signed-Distance-Properties}, that $U_{\tau/C_2} \subset \tilde U$, and that $\LL^n(\tilde U \backslash U_{\tau/C_2}) \leq C_4\tau$. Hence using \eqref{eqn:iso-func} and Lemma \ref{Lem:Isomorphism} we find that
\begin{displaymath}
  \mathcal{I}_{\tilde U}(\V) \geq \inf_{\eta \le C_4 \tau}\mathcal{I}_{U_{\tau/C_2}}(\V-\eta) \geq C(\V-C_4\tau)^{(n-1)/n}
\end{displaymath}
as long as $\V - C_4\tau \in (C\tau,A)$.

Again recalling that $\LL^n(E_{\hat \V} \backslash E_0) \in [\delta, \frac{3 \delta}{2}]$ we find that, for $\delta$ sufficiently small and $\tau = c\delta$ with sufficiently small $c>0$,
\[
P(E_{\hat \V} ; \tilde U) \geq C \delta^{(n-1)/n}.
\]
This inequality, together with \eqref{Eqn:numerator-bound} and \eqref{Eqn:denominator-bound}, proves \eqref{Eqn:Second-Derivative-Bound}.

%
%
%
%
%

By then using an argument as in the proof of Lemma 2.7 in \cite{SternbergZumbrunIsoPer} (see also \cite{MurrayRinaldi}) we find that $\Iloc$ is semi-concave on $J_\delta$, which is the desired conclusion.
\end{proof}

As $\Iloc$ is semi-concave, it has a left derivative $(\Iloc)_-'$ and a right derivative $(\Iloc)_+'$ at every point $\V$  in $J_\delta$, with $(\Iloc)_-' (\V)\geq (\Iloc)_+'(\V)$. Furthermore, by considering a normal perturbation of $E_0$, we have that $(n-1)\kappa_{E_0} \in [(\Iloc)_-'(\V_0),(\Iloc)_+'(\V_0)]$. The following result gives a simple, yet important observation.

\begin{proposition}\label{lem:delta-to-zero}
Assume that $\Omega$ satisfies \eqref{domainAssumptions}  and let $E_0\subset \Omega$ be a volume-constrained local perimeter minimizer in $\Omega$, with $  \LL^n(E_0)=\V_m$. Then as $\delta \to 0$, $(\Iloc)_-'(\V_m) \to (n-1)\kappa_{E_0}$ and $(\Iloc)_+'(\V_m) \to (n-1) \kappa_{E_0}$, where $\kappa_{E_0}$ is the mean curvature of $E_0$.
\end{proposition}

\begin{proof}
We will prove the result for the left derivative. For any fixed $\delta$, pick a sequence of points $\V_k \uparrow \V_0$ at which $\Iloc$ is differentiable. This is possible as $\Iloc$ is semi-concave. Also, as $\Iloc$ is semi-concave we have that $(\Iloc)'(\V_k) \to (\Iloc)_-'(\V_0)$. Let $E_{\V_k}$ be a minimizer of
\[
\min \{ P(E;\Omega)  :\, E \subset \Omega \text{ Borel, } \LL^n(E) = \V_k,\, \alpha(E_0,E) \leq \delta  \}.
\]

We claim that there exists a volume-constrained perimeter minimizer $E_0^\delta$, satisfying $\alpha(E_0^\delta, E_0) \leq \delta$, $ \LL^n(E_0^\delta) = \V_0$, and with mean curvature $\kappa_0^\delta = (\Iloc)_-'(\V_0)$.

First, suppose that we can pick a subsequence of $\V_k$ (not relabeled), such that $$\min\{ \mathcal{L}^{n}(E_{\V_k} \backslash E_0), \mathcal{L}^{n}(E_0\backslash E_{\V_k})\} = \mathcal{L}^{n}(E_{\V_k} \backslash E_0).$$ 
Suppose furthermore that $\liminf_{k \to \infty} \mathcal{L}^n(E_0\backslash E_{\V_k})  \geq \delta$.

Under these assumptions, and as long as $\delta$ is small enough and $\V_k$ is close enough to $\V_0$, we have that $\partial E_{\V_k} \cap E_0$ is a non-empty set. Furthermore, taking  local variations with support in $E_0$ will not increase the value of $\alpha(E_{\V_k},E_0)$. Hence, the mean curvature of $ E_{\V_k}$ inside the set $E_0$, which we will denote $\kappa_{\delta,k}^*$, will satisfy ( see Chapter 17 in \cite{MaggiBook})
\[
(n-1)\kappa_{\delta,k}^* = (\Iloc)'(\V_k).
\]
We remark that since $(\Iloc)'(\V_k) \to (\Iloc)_-'(\V_0)$, we immediately have that $\kappa_{\delta,k}^*$ is bounded.

By BV compactness, $\chi_{E_{\V_k}} \to \chi_{E_0^\delta}$ in $L^1(\Omega)$, for some set $E_0^\delta$ with finite perimeter. By lower semi-continuity of the perimeter, we have that $P(E_0^\delta;\Omega) = \Iloc(r_0) \leq P(E_0;\Omega)$. As $E_0$ is a local volume-constrained perimeter minimizer, for $\delta$ small enough we have that $E_0^\delta$ is a local volume-constrained perimeter minimizer as well. In particular, $\partial E_0^\delta$ is a surface of constant mean curvature. Furthermore, by the assumption that $\liminf_{k \to \infty} \mathcal{L}^n(E_0\backslash E_{\V_k})  \geq \delta$ we know that $\partial E_0^\delta \cap E_0$ is a set with positive perimeter.

By using the uniform bound on the curvatures, along with elliptic regularity, we then have that $E_{\V_k} \to E_0^\delta$ in $C^\infty$ in compact subsets of $E_0$ (see the proof of Theorem 1.9 in \cite{MurrayRinaldi}). Hence the mean curvature $\kappa_0^\delta$ of $E_0^\delta$ satisfies $(n-1)\kappa_0^\delta = (\Iloc)_-'(\V_0)$.

The case where $\liminf_{k \to \infty} \mathcal{L}^n(E_0\backslash E_{\V_k})  < \delta$ is in fact simpler, because the $\alpha$-constraint will not be saturated and any local perturbation is permissible. On the other hand, if we cannot pick a subsequence of $\V_k$ satisfying $\min\{ \LL^n(E_{\V_k} \backslash E_0), \LL^n(E_0 \backslash E_{\V_k})\} = \LL^n(E_{\V_k} \backslash E_0)$, then we must be able to pick a subsequence satisfying $\min\{ \LL^n(E_{\V_k} \backslash E_0), \LL^n(E_0 \backslash E_{\V_k})\} = \LL^n(E_{0} \backslash E_{\V_k})$. We then conduct the same steps, but this time in $\Omega \backslash E_0$. This proves the claim.

Finally, we recall that $\alpha(E_0^\delta, E_0) \leq \delta$. Hence we have that $\chi_{E_0^\delta} \to \chi_{E_0}$ in $L^1(\Omega)$ as $\delta \to 0$. By again using the same argument, $E_0^\delta$ must in fact converge in $C^\infty$ to $E_0$, and hence $\kappa_0^\delta \to \kappa_{E_0}$, or in other words, $(\Iloc)_-' \to (n-1)\kappa_{E_0}$. This concludes the proof.
\end{proof}

\section{Rearrangements and Weighted Problem}

Let $I=(A,B)$ for some $A<B$ and consider a function $\eta: I \to [0,\infty)$ which satisfies the following:
\begin{align}
\label{eta1} \eta &\in C(I) \cap C^1((A,t_0]) \cap C^1([t_0,B)),\qquad \eta>0 \quad\text{in }I\\ 
\label{eta2}
d_1(t-A)^{\frac{n-1}{n}} &\leq \eta(t) \leq d_2(t-A)^{\frac{n-1}{n}} \text{ for } t \in (A, A+t^\star),\\
\label{eta3}
d_3(B-t)^{\frac{n-1}{n}} &\leq \eta(t) \leq d_4(B-t)^{\frac{n-1}{n}} \text{ for } t \in (B-t^\star, B), \\
\label{eta4}
|\eta'(t)| &\leq \frac{d_5\eta(t)}{\min\{B-t,t -A\}}\quad\text{for }t\in I\setminus\{t_0\}, \quad \eta'_-(t_0) \geq \eta'_+(t_0),\\
\label{eta5}
\int_A^{t_0} \eta \dt &= \V_m,\qquad \int_I \eta \dt = 1,
\end{align}
for some $A<t_0<B$ and for some constants $d_1$, \ldots, $d_5>0$ and $t^\star>0$.

Next, define the energy
\[
G_\e(v) := \begin{cases} \int_I (W(v) + \e^2|v'|^2)\eta \dt &\text{ if } v \in H_\eta^1(I) \text{ and }\int_I v \eta \dt = m, \\ \infty &\text{ otherwise}. \end{cases}
\]
Under the hypotheses \eqref{eta1}--\eqref{eta5}, 
following the proof of Theorem 4.4 in \cite{LeoniMurray}, it can be shown that $G_\e^{(1)}=\e^{-1}G_\e \gto  G_0^{(1)}$, where $G_0^{(1)}$ is given by
\[
G_0^{(1)}(v) := \begin{cases} \frac{2c_W}{b-a} |Dv|_\eta(I) \text{ if } v \in BV_\eta(I) \text{ and } \int_I v \eta \dt = m,\\ \infty \text{ otherwise,} \end{cases}
\]
with $c_W$ the constant given in \eqref{c0Definition}. In view of \eqref{partitionProblemMassConstraint} and \eqref{eta5}, it can also be shown as in Theorem 4.6 in \cite{LeoniMurray} that $v_0 =  a\chi_{[A,t_0)} + b\chi_{[t_0,B]}$ is an isolated $L^1$-local minimizer of $G_0^{(1)}$, and hence for some $\hat \delta$ sufficiently small we have that $v_0$ is the unique limit of minimizers $v_\e$ of the functionals
\[
J_\e(v) := \begin{cases} G_\e(v) &\text{ if } v \in H^1_\eta(I),\, \int_I v \eta \dt = m \text{ and } \|v-v_0\|_{L_\eta^1} \leq \hat \delta, \\ \infty &\text{ otherwise}.  \end{cases}
\]
Note that $v_\e$ satisfies the Euler-Lagrange equation
\begin{equation}\label{E-L}
2\e^2(v_\e' \eta)' - W'(v_\e) \eta = \e \lambda_\e \eta.
\end{equation}

Our goal is to prove the following theorem:

\begin{theorem}\label{thm:1D}
Assume that $W$ satisfies hypotheses \eqref{W_Smooth}--\eqref{WGurtin_Assumption} and that $\eta$ satisfies \eqref{eta1}--\eqref{eta5}. Let $v_\e$ be a minimizer of $G_\e$ with $v_\e \to v_0$ in $L_\eta^1$ as $\e\to 0^+$. Then, 
\begin{align}
\liminf_{\e \to 0^+}& \frac{G_\e^{(1)}(v_\e)-2c_W \eta(t_0)}{\e} \geq
2 \eta_-'(t_0) \int_{-\infty}^0 W^{1/2} (z(s-\tau_0)) z'(s-\tau_0) s\ds\\ &+ 2 \eta_+'(t_0)  \int_0^\infty W^{1/2} (z(s-\tau_0)) z'(s-\tau_0) s\ds+\begin{cases}
\frac{\lambda_0^2}{2W''(a)}\int_I \eta(t) \dt & \text{if }q=1,\\
0&\text{if }q<1,
\end{cases} 
\label{liminf}
\end{align}
where  $c_{sym}$ is the constant given in \eqref{c1Definition},
\begin{equation}
\lim_{j \to \infty}\lambda_{\e_j} =\lambda_0 \in \left[\frac{2c_W\eta'_+(t_0)}{(b-a) \eta(t_0)},\frac{2c_W\eta'_-(t_0)}{(b-a) \eta(t_0)}\right]
\label{lambda e}
\end{equation}
for some subsequence $\e_j\to 0^+$, and the number $\tau_0$ is given by
\begin{equation}\label{tau 0}
\eta(t_0)\int_\R z(s-\tau_0) - sgn_{a,b}(s) \ds = \frac{\lambda_0}{W''(a)} \int_I \eta (t)\dt,
\end{equation}
with $z$ the solution to \eqref{profileCauchyProblem}.
\end{theorem}

\begin{proof} By taking a subsequence (not relabeled),
without loss of generality, we may assume that the $\liminf$ on the left-hand side of \eqref{liminf} is actually a limit. Also, for simplicity we take $t_0=0$.

\noindent{\bf Step 1.} We claim that \eqref{lambda e} holds. This proof follows as in Theorem 4.9 in \cite{LeoniMurray}. The only difference is that at the last part of the proof we can no longer use the fact that $\eta$ is of class $C^1$ and we need to show that
\[
\lim_{\e \to 0^+} \int_I W^{1/2} (v_\e) |v'_\e| \eta' \psi \dt = c_W \psi(0) \eta_1,
\]
for some $\eta_1 \in [\eta'_+(0), \eta'_-(0)]$. Following the proof cited above, we know that $W^{1/2}(v_\e)|v_\e'|\eta\mathcal{L}^1\lfloor [A,B] \stackrel{*}{\rightharpoonup} c_W\eta(0)\delta_{0}$. Hence by picking an appropriate subsequence, we have, for some $\theta \in [0,1]$,
\begin{align*}
&W^{1/2}(v_\e)|v_\e'|\eta\mathcal{L}^1\lfloor [A,0] \stackrel{*}{\rightharpoonup} \theta c_W\eta(0)\delta_{0}, \\
&W^{1/2}(v_\e)|v_\e'|\eta\mathcal{L}^1\lfloor [0,B] \stackrel{*}{\rightharpoonup} (1-\theta) c_W\eta(0)\delta_{0}.
\end{align*}
Hence, 
\[
\lim_{\e \to 0^+} \int_I W^{1/2} (v_\e) |v'_\e| \eta' \psi \dt = c_W \psi(0) (\theta \eta'_-(0) + (1-\theta) \eta'_+(0)),
\]
which is the desired conclusion.

\noindent{\bf Step 2.}  We claim that there exists a sequence of numbers $\tau_\e \to \tau_0$, where $\tau_0$ is given in \eqref{tau 0}, 
 so that the functions $w_\e(s): = v_\e\left(\e s \right)$, $s\in (A\e^{-1},B\e^{-1})$, converge weakly to the profile $w_0 := z(\cdot-\tau_0)$ in $H^1((-l,l))$ for any fixed $l > 0$, and satisfy
\[
w_\e(\tau_\e) = c_\e,
\]
where $c_\e$ is the central zero of $W' + \e \lambda_\e$. 

This follows from the proofs of Lemmas 4.18 and 4.19 in \cite{LeoniMurray} (see also \cite{DalMasoFonsecaLeoni}). We note that those proofs use significant machinery from that work, including detailed decay estimates, but do not require anything more than a Lipschitz estimate on $\eta$ near $0$ and \eqref{eta2}, \eqref{eta3}, and \eqref{eta4}.

\noindent{\bf Step 3.}   We claim that \eqref{liminf} holds. Define $\eta_\e(s): = \eta(s \e)$, $s\in (A\e^{-1},B\e^{-1})$. After changing variables, and setting $l_\e := C |\log \e|$, we obtain
\begin{align*}
G_\e^{(1)}(v_\e) &= \e^{-1}\int_{-l_\e}^{l_\e} (W^{1/2}(w_\e) - w_\e')^2 \eta_\e \ds +2\e^{-1}\int_{-l_\e}^{l_\e} W^{1/2}(w_\e) w_\e' (\eta_\e - \eta(0))  \ds \\
        &\quad+ \e^{-1}\int_{[A \e^{-1},B \e^{-1}] \backslash (-l_\e,l_\e)} \left(W(w_\e) + (w_\e')^2\right) \eta_\e \ds + \e^{-1}2\eta(0)\left(\int_{-l_\e}^{l_\e} W^{1/2}(w_\e) w_\e'  \ds - c_W\right) \\
        &\geq 2\e^{-1}\int_{-l_\e}^{l_\e} W^{1/2}(w_\e) w_\e' (\eta_\e - \eta(0))  \ds \\
        &\quad+ \e^{-1}\int_{[A \e^{-1},B \e^{-1}] \backslash (-l_\e,l_\e)} W(w_\e) \eta_\e \ds + \e^{-1}2\eta(0)\left(\int_{-l_\e}^{l_\e} W^{1/2}(w_\e) w_\e'  \ds - c_W\right) .
\end{align*}
The last term goes to zero, see 4.105 in \cite{LeoniMurray}. Following the proof of 4.106 in \cite{LeoniMurray}, the second to last term satisfies
\[
\lim_{\e \to 0^+} \e^{-1}\int_{[A \e^{-1},B \e^{-1}] \backslash (-l_\e,l_\e)} W(w_\e) \eta_\e \ds =\begin{cases}
\frac{\lambda_0^2}{2W''(a)}\int_I \eta \dt & \text{if }q=1,\\
0&\text{if }q<1.
\end{cases} 
\]

Finally, by \eqref{eta1} the function $\eta$ satisfies the following Taylor's formula: 
\[
\eta(t)= \eta(0) - t^- \eta'_-(0) + t^+ \eta'_+(0)|+ |t|R_1(t),\]
where $R_1(t)\to 0$ as $t\to 0$. Hence, we find that \begin{align*}
& 2\e^{-1}\int_{-l_\e}^{l_\e} W^{1/2}(w_\e(s)) w_\e' (s)(\eta_\e(s) - \eta(0))  \ds= 2\int_{-l_\e}^{l_\e} W^{1/2}(w_\e(s)) w_\e'(s)(- \eta'_-(0) s^- + \eta'_+(0) s^+) \ds
\\&+2\int_{-l_\e}^{l_\e} W^{1/2}(w_\e(s)) w_\e'(s)|s|R_1(\e s) \ds.
\end{align*}
As in \cite{LeoniMurray}, we now break the integrals over $[-l_\e,-l]$, $[-l,l]$, $[l,l_\e]$ for any fixed $l > 0$. Since by Step 2, 
$\{w_\e\}$ converges weakly to $z(\cdot-\tau_0)$ in $H^1((-l,l))$, we can follow the computations after formula (4.106) in \cite{LeoniMurray} using the exponential decay (see (4.95) and (4.96) in \cite{LeoniMurray}) in $[-l_\e,-l]$ and $[l,l_\e]$ to obtain that  
\begin{align*}
&\lim_{\e \to 0^+} 2\int_{-l_\e}^{l_\e} W^{1/2}(w_\e(s)) w_\e'(s)(- \eta'_-(0) s^- + \eta'_+(0) s^+) \ds \\
&= 2 \eta_-'(0) \int_{-\infty}^0 W^{1/2} (z(s-\tau_0)) z'(s-\tau_0) s\ds + 2 \eta_+'(0)  \int_0^\infty W^{1/2} (z(s-\tau_0)) z'(s-\tau_0) s\ds.
\end{align*}
Similarly, using the facts that $R_1(t)\to 0$ as $t\to 0$ and $\e |s|\le \e l_\e\le C\e|\log\e|$ for $|s|\le l_\e$, we can use Step 2 to show that  
\[\lim_{\e \to 0^+}2\int_{-l}^{l} W^{1/2}(w_\e(s)) w_\e'(s)|s|R_1(\e s) \ds=0,\]
while by Lemma 4.19 and (4.96) in \cite{LeoniMurray},
\begin{align*}
2\int_{l}^{l_\e} W^{1/2}(w_\e(s))| w_\e'(s)||s||R_1(\e s)| \ds\le 2\Vert R_1\Vert_{L^\infty(-\e l_\e, \e l_\e)} \Vert w_\e'\Vert_{\infty}
\int_{l}^{l_\e} W^{1/2}(w_\e(s)) |s| \ds\to 0
\end{align*}
as $\e\to 0^+$. A similar estimate holds in  $[-l_\e,-l]$. This  concludes the proof of \eqref{liminf}.

\end{proof}

\section{Proof of the Main Results}

Now we give a proof of Theorem \ref{theorem main}.

\begin{proof}[Proof of Theorem \ref{theorem main}] We only give the proof in the case $q=1$ in \eqref{WPrime_At_Wells}, the case $q<1$ being similar. 
Since $\mathcal{I}_\Omega\le \Iloc$ (see \eqref{eqn:iso-func} and \eqref{labelstar3}), reasoning as in Proposition 3.1 in \cite{LeoniMurray} we can construct a function $\mathcal{I}\in C(0,1)\cap C^1((0,\V_m])\cap C^1([\V_m,0))$  satisfying 
\begin{align}
\label{I1}&\Iloc\ge \mathcal{I}>0\text{ in }(0,1), \\
\label{I2}&\mathcal{I}(\V_m) = \Iloc(\V_m), \qquad 
\mathcal{I}'_\pm(\V_m) = (\Iloc)'_\pm(\V_m),\\
\label{I3}&\mathcal{I}(\V)=C_0\V^{\frac{n-1}{n}}\text{ for }\V\in (0,r),\qquad \mathcal{I}(\V)=C_0(1-\V)^{\frac{n-1}{n}}\text{ for }\V\in (1-r,1)
\end{align}	
for some constant $C_0>0$ and some $0<r<1/2$ small.	
Let $\eta := \mathcal{I} \circ V_\Omega$, where $V_\Omega$ satisfies
\begin{equation} \label{function V}
\frac{d}{dt} V_\Omega(t) = \mathcal{I}(V_\Omega(t)),\quad V_\Omega(0) = \V_m.
\end{equation}
As in the proof of Theorem 5.1 in \cite{MurrayRinaldi} one can show that $\eta$ satisfies all of the assumptions \eqref{eta1}--\eqref{eta5}.
	
Let $u_\e$ be a minimizer of $\mathcal{F}_\e$ and let 
$v_\e:=f_{u_\e}$ be the increasing function given in Remark 3.11 of \cite{LeoniMurray}.  
Following the proof of Theorem 5.1 in \cite{LeoniMurray}  (see also \cite{MurrayThesis} or \cite{MurrayRinaldi} for more details), we have that
\[
\frac{\mathcal{F}^{(1)}_\e(u_\e) - \min \mathcal{F}_0}{\e} \geq \frac{G^{(1)}_\e(v_\e) - 2c_W \eta(t_0)}{\e},
\]

Hence, by Theorem \ref{thm:1D}, we have that
\begin{align}
\liminf_{\e\to 0^+} \frac{\mathcal{F}^{(1)}_\e(u_\e) - \min \mathcal{F}_0}{\e} \geq \frac{\lambda_0^2(\delta)}{2W''(a)} \int_I \eta(t) \dt& +
2 \eta_-'(t_0) \int_{-\infty}^0 W^{1/2} (z(s-\tau_0(\delta))) z'(s-\tau_0(\delta)) s\ds\\ &+ 2 \eta_+'(t_0)  \int_0^\infty W^{1/2} (z(s-\tau_0(\delta))) z'(s-\tau_0(\delta)) s\ds,
\end{align}
where 
\begin{equation}\label{lambda delta}
\lambda_0(\delta) \in \left[\frac{2c_W\eta'_+(t_0)}{(b-a) \eta(t_0)},\frac{2c_W\eta'_-(t_0)}{(b-a) \eta(t_0)}\right]
\end{equation}
and $\tau_0(\delta)$ is given by
\begin{equation}\label{tau delta}
\eta(t_0)\int_\R z(s-\tau_0(\delta)) - sgn_{a,b}(s) \ds = \frac{\lambda_0(\delta)}{W''(a)} \int_I \eta(t) \dt.
\end{equation}
 By Proposition \ref{lem:delta-to-zero}, \eqref{I2}, and \eqref{function V}, we find that as $\delta \to 0$ the quantities $\eta'_-(t_0)$ and $\eta'_+(t_0)$ converge to the same value, namely, $(n-1)\kappa_{E_0}$, and hence by \eqref{lambda delta} and \eqref{tau delta}  we have that $\lambda_0(\delta) \to \lambda_u$ and $\tau_0(\delta) \to \tau_u$ converge as well. Thus by taking $\delta \to 0$ we obtain
\[
\liminf_{\e \to 0^+} \frac{\mathcal{F}^{(1)}_\e(u_\e) - \min \mathcal{F}_0}{\e} \geq \frac{2c_W^2(n-1)^2}{W''(a) (b-a)^2}\kappa_u^2 + 2(c_{sym} + c_W\tau_u) (n-1) \kappa_u P(\{u=a\};\Omega),
\]
which is the desired result.
\end{proof}

The proofs of Theorems \ref{theorem ac} and \ref{globalMotion} now follow from Theorems 1.2 and 1.7
and Theorem 1.4 in \cite{MurrayRinaldi}, respectively, with the only change that we apply Theorem \ref{theorem main} of this paper in place of Theorem 1.1. in \cite{LeoniMurray}.

\section*{Acknowledgments}
The authors warmly thank the Center for Nonlinear Analysis,  where part of this work was carried out.

\section*{Compliance with Ethical Standards}
Part of this research was carried out at Center for Nonlinear Analysis. The center is partially
supported by NSF Grant No. DMS-0635983 and NSF PIRE Grant No. OISE-0967140. 
The research of G. Leoni was partially funded by the NSF under Grant No. and DMS-1412095 and the one of R. Murray by NSF PIRE Grant No. OISE-0967140.

\bibliography{ModicaMortolaRefs}
\bibliographystyle{acm}

\end{document}